\documentclass [10pt]{article}

\usepackage{amssymb}
\usepackage{amsfonts}
\usepackage{graphicx}
\usepackage{amsmath}
\usepackage{amsmath,amsfonts,amssymb}
\usepackage{dsfont}
\usepackage{graphics}
\usepackage{subcaption}
\usepackage{mathtools}
\usepackage{commath}
\usepackage[colorlinks=true]{hyperref}

\usepackage{color}

\def\[{\left[}
\def\]{\right]}
\def\({\left(}
\def\){\right)}

\newcommand{\RR}{\R}

\newcommand{\N}{\mathbb{N}}
\newcommand{\R}{\mathbb{R}}

\newtheorem{theorem}{Theorem}[section]

\newtheorem{assumption}[theorem]{Assumption}

\newtheorem{corollary}[theorem]{Corollary}

\newtheorem{definition}[theorem]{Definition}

\newtheorem{lemma}[theorem]{Lemma}

\newtheorem{proposition}[theorem]{Proposition}
\newtheorem{remark}[theorem]{Remark}

\newenvironment{proof}[1][Proof]{\noindent\textit{#1.} }{\hfill \rule{0.5em}{0.5em}}

\renewcommand{\eqref}[1]{(\ref{#1})}

\newcommand{\be}{\beta}
\renewcommand{\R}{\mathbb{R}}
\newcommand{\La}{\Lambda}
\newcommand{\la}{\lambda}
\newcommand{\ga}{\gamma}

\newcommand{\ep}{\epsilon}

\renewcommand{\RR}{{\rm I\kern -1.6pt{\rm R}}}

\numberwithin{equation}{section}
\numberwithin{figure}{section}
\allowdisplaybreaks

\begin{document}

\title{\textbf{Traveling wave solutions and spreading speeds for a scalar age-structured equation with nonlocal diffusion\thanks{Research was partially supported by the National Natural Science Foundation of China (No. 12301259 and No. 12371169).}}}

\author{{\sc Arnaud Ducrot$^{a}$ and Hao Kang$^{b}$}\\[2mm]
	{\small $^a$Normandie Univ, UNIHAVRE, LMAH, FR-CNRS-3335, ISCN, 76600 Le Havre, France}\\
	{\small Emails: arnaud.ducrot@univ-lehavre.fr}\\
	{\small $^b$Center for Applied Mathematics, Tianjin University, Tianjin, China} \\
	{\small Email: haokang@tju.edu.cn}
}
\date{}
\maketitle

{\centering Dedicated to Professor Shigui Ruan for His 60th Birthday.\par}

\begin{abstract}
	In this paper, we study the existence of traveling wave solutions and the spreading speed for the solutions of an age-structured epidemic model with nonlocal diffusion. Our proofs make use of the comparison principles both to construct suitable sub/super-solutions and to prove the regularity of traveling waves solutions.
	
	\vspace{0.3cm} \noindent {\em Key words:} Age structure; Nonlocal diffusion; Traveling wave solutions; Spreading speeds.
	
	\vskip 0.2cm
	
	
\end{abstract}	

\section{Introduction and main results}

In this paper, we are concerned with the study of an age dependent epidemic model with nonlocal diffusion in space, motivated by \cite{ducrot2007travelling}, where the first author studied as similar problem with random diffusion in space. Our aim is to study a classical SI model from the point of view of the spatial spread of an epidemics. Here $a\in(0, a^+)$ denotes the physiological age and $a^+\in(0, \infty)$ is the maximum age of an individual, $t>0$ denotes the time and $x\in \R$ is the spatial position of an individual. The population can be split into two sub-populations, the susceptible and the infective. We denote by $S(t, a, x)$, respectively $I(t,a,x)$, the age distribution of the susceptible individuals, respectively the infective individuals, at time $t$ and spatial location $x\in \R$. With these notations, the model that we consider reads as follows

\begin{equation}\label{original}
	\begin{cases}
		\left(\frac{\partial I}{\partial t}+\frac{\partial I}{\partial a}\right)(t, a, x)=(J\ast I-I)(t, a, x)+S(t, a, x)\int_0^{a^+}K(a, a')I(t, a', x)da'-\mu(a)I(t, a, x),\\
		\left(\frac{\partial S}{\partial t}+\frac{\partial S}{\partial a}\right)(t, a, x)=d\left(J\ast S-S\right)(t, a, x)-S(t, a, x)\int_0^{a^+}K(a, a')I(t, a', x)da'-\mu(a)S(t, a, x),\\
		I(t, 0, x)=p\int_0^{a^+}\beta_2(a)I(t, a, x)da,\\
		S(t, 0, x)=\int_0^{a^+}\beta_1(a)S(t, a, x)da+(1-p)\int_0^{a^+}\beta_2(a)I(t, a, x)da,\\
		I(0, a, x)=I_0(a, x),\;S(0, a, x)=S_0(a, x),
	\end{cases}
\end{equation} 
wherein we have set
$$
\left(J\ast u-u\right)(x):=\int_\R J(x-y) u(y)dy-u(x), \quad \forall u\in C_b(\R),
$$
for some continuous convolution kernel $J:\R\to\R$, whose specific properties will be presented in Assumption \ref{Assump} below, while $C_b(\R)$ denotes the space of bounded and continuous functions. Here the function $K(a, a')$ denotes the rate of the disease transmission from infective individuals of age $a'$ to susceptible individuals of age $a$. In addition, the function $\mu$ denotes the age-specific death rate for both the infective and susceptible individuals, $\be_1$ and $\be_2$ denote the age-specific birth rate for the susceptible and infective ones, respectively. Moreover, we assume that these birth rates are identical, that is $\beta_1\equiv\beta_2=\beta$. The constant $p\in[0, 1]$ denotes the proportion of the vertical transmission, that is the proportion of the infective newborns inherited from their infective parents. Further, we assume that $S$ and $I$ have the same diffusion coefficient, that is $d=1$. This assumption combined with no additional death rate due to the disease allow us to reduce the system to a single scalar equation. We would like to mention that for the existence of traveling wave of age-structured $SI$ system with random diffusion have been obtained by Ducrot and Magal \cite{ducrot2009travelling,ducrot2011travelling} with/without external supply and Ducrot et al. \cite{ducrot2010travelling} in a multigroup framework, respectively. 

In this work, we consider that the total population has a demographic equilibrium, meaning that birth and death equilibrate the population. This prevents the population from going to extinction and from exploding as time increases. The total population stabilizes to a steady state as $t\to\infty$. Mathematically speaking, this assumption can be written as the following condition on the demographic functions $\beta$ and $\mu$:
\begin{equation}\label{demography}
\int_0^{a^+}\beta(a)\exp\left(-\int_0^a\mu(a')da'\right)da=1.
\end{equation}  

This condition may imply that multiple steady states exist for the total population $I+S$ when $d=1$ (see Webb \cite{webb1984theory} for some discussions on this condition). Next, we are looking for heteroclinic solutions of system \eqref{original} with the following behavior for $x\to\pm\infty$:
\begin{equation}
\begin{array}{ll}
I(t, a, -\infty)=\exp\left(-\int_0^a\mu(a')da'\right),\; I(t, a, +\infty)=0,\\
S(t, a, -\infty)=0, \; S(t, a, +\infty)=\exp\left(-\int_0^a\mu(a')da'\right)\label{SI-infty},
\end{array}
\end{equation}
These conditions mean that at $x=+\infty$, the population is only composed of the susceptible, whereas at $x=-\infty$ the population is only composed of infected individuals. \\
Now when $d=1$, adding-up the first two equations in \eqref{original}, the nonlinear terms cancel and one obtains using \eqref{SI-infty}, that
$$
(I+S)(t, a, x)=\exp\left(-\int_0^a\mu(a')da'\right):=\pi(a), \text{ provided }(I_0+S_0)(a, x)=\pi(a).
$$
As a consequence, system \eqref{original}-\eqref{SI-infty} reduces to the scalar equation for the unknown function $I$
\begin{equation}\label{I-equation}
	\begin{cases}
		\left(\frac{\partial I}{\partial t}+\frac{\partial I}{\partial a}\right)(t, a, x)=(J\ast I-I)(t, a, x)+\int_0^{a^+}K(a, a')I(t, a', x)da'\,(\pi(a)-I(t, a, x))-\mu(a)I(t, a, x),\\
		I(t, a, -\infty)=\pi(a),\; I(t, a, \infty)=0,\\
		I(t, 0, x)=p\int_0^{a^+}\beta(a)I(t, a, x)da.
	\end{cases}
\end{equation}
We rewrite this problem with the new unknown function $u(t, a, x)$ defined by
$$
u(t, a, x)=\frac{I(t, a, x)}{\pi(a)}
$$
and we obtain the following equation for the function $u$:
\begin{equation}\label{u-equation}
	\begin{cases}
		\left(\frac{\partial u}{\partial t}+\frac{\partial u}{\partial a}\right)(t, a, x)=(J\ast u-u)(t, a, x)+\int_0^{a^+}K(a, a')\pi(a')u(t, a', x)da'\,(1-u(t, a, x)),\\
		u(t, a, -\infty)=1,\; u(t, a, \infty)=0,\\
		u(t, 0, x)=p\int_0^{a^+}\beta(a)\pi(a)u(t, a, x)da,
	\end{cases}
\end{equation} 
Now we set 
$$
\ga(a):=\beta(a)\pi(a),
$$
and recalling  condition \eqref{demography}, we assume the following properties. 
\begin{assumption}\label{Assump}
{\rm We assume that the following properties hold true:
	\begin{itemize}
		\item [(i)] the function $\gamma$ is continuous and nonnegative in $[0, a^+]$ and satisfies $\int_0^{a^+}\gamma(a)da=1$; 
		
		\item [(ii)] the function $\pi$ is continuous, positive in $[0, a^+]$ and the function $\gamma\pi$ does not identically vanish on $(0, a^+)$;

        \item [(iii)] the kernel $J$ is continuous and nonnegative with $\int_\R J(x)dx=1, J(0)>0$ and $J(-x)=J(x), \forall x\in\R$. Moreover, $\int_\R J(x)e^{l x}dx<\infty$ for any $l>0$;
        
        \item [(iv)] the transmission rate $K(\cdot, \,\cdot)$ is continuous and positive in $[0, a^+]\times[0, a^+]$.

    \end{itemize}
    }
\end{assumption}

We now consider equation \eqref{u-equation} and we are concerned with traveling wave solutions, that is particular solutions of the form $U(a, \xi)=u(a, x-ct)$ with $\xi=x-ct$. Here $c$ is an unknown real number which should be found together with the unknown function $U$. Before proceeding, we set the rate of vertical transmission to be one, i.e. $p=1$. This setting is kind of technical mathematically, since $p=1$ is used to guarantee that (1) the limiting equation \eqref{equi2} has only two solutions $0$ and $1$, see Section \ref{LB}; (2) the hair trigger effect holds, see Lemma \ref{HTE}. 
Next, using the moving frame, that is the variable $\xi=x-ct$, the equation for the profile function $U$ becomes for $\xi\in \R$ and $a\in [0,a^+]$:
\begin{equation}\label{U-equation}
	\begin{cases}
		\frac{\partial U(a, \xi)}{\partial a}=\left(J\ast_\xi U-U\right)(a, \xi)+c\frac{\partial U(a, \xi)}{\partial \xi}+\int_0^{a^+}K(a, a')\pi(a')U(a', \xi)da'\,(1-U(a, \xi)),\\
		U(a, -\infty)=1,\; U(a, \infty)=0, \text{ uniformly in $[0, a^+]$},\\
		U(0, \xi)=\int_0^{a^+}\gamma(a)U(a, \xi)da.
	\end{cases}
\end{equation}
When the age specific demographic and epidemic functions are ignored, this equation reduces to a monostable equation with nonlocal diffusion, which has been well studied by many researchers, see for example Coville et al. \cite{coville2005propagation,coville2007non,coville2008existence,coville2008nonlocal}, Fang and Zhao \cite{fang2014traveling}, Li et al. \cite{li2010entire,sun2011traveling} and Shen et al. \cite{shen2010spreading,shen2015transition} and the references therein. In addition, we mention that the global dynamics of \eqref{I-equation} with spatially dependent coefficients on the bounded domain can be studied using the sign of the spectral bound of a linearized operator at zero, the interested readers can refer to Ducrot et al. \cite{Ducrot2022Age-structuredI,Ducrot2022Age-structuredII} and Kang and Ruan \cite{kang2021mathematical} for more details. 
 
Next, we solve \eqref{U-equation} formally along characteristic lines. Define the characteristic line $h(a, \xi)$ as the solution of the following equation,
\begin{equation}\label{characteristic}
\begin{cases}
\frac{\partial h(a, \xi)}{\partial a}=-c, &a\in[0, a^+],\;\xi\in\R,\\
h(0, \xi)=\xi,
\end{cases}
\end{equation}
that reads as $h(a, \xi)=\xi-ca$, which implies that $\partial_\xi h(a, \xi)=1$.
Next, solving \eqref{U-equation} along the characteristic line $h(a, \xi)$, one obtains that the function
\begin{equation}\label{w}
	w(a, \xi)=U(a, h(a, \xi))=U(a, \xi-ca)
\end{equation}
satisfies the following equation,
\begin{equation}\label{w-equation}
	\begin{cases}
		\frac{\partial w(a, \xi)}{\partial a}=\left(J\ast_\xi w-w\right)(a, \xi)+\int_0^{a^+}K(a, a')\pi(a')w(a', \xi)da'\,(1-w(a, \xi)),\\
		w(a, -\infty)=1,\; w(a, \infty)=0,\text{ uniformly in $[0, a^+]$},\\
		w(0, \xi)=\int_0^{a^+}\gamma(a)w(a, \xi+ca)da.
	\end{cases}
\end{equation}
The above reformulation using the characteristic lines allows us to take care of the term $\frac{\partial U}{\partial\xi}$ appearing in \eqref{U-equation}.  We now give the definition of a traveling wave solution of \eqref{u-equation} as follows.
\begin{definition}\label{TW}
	A function $w=w(a,\xi)$ is said to be a traveling wave solution of \eqref{u-equation} if $w\in C([0, a^+]\times\R)$ is a bounded solution of \eqref{w-equation} such that for any $a\in(0, a^+)$ the map $\xi\mapsto w(a, \xi)$ is globally Lipschitz continuous and for any $\xi\in\R$, $a\mapsto w(a, \xi)\in W^{1, 1}(0, a^+)$.
\end{definition}
The first aim of this paper is to prove the existence of traveling wave solutions for \eqref{u-equation}. More precisely, we will prove the following result.
\begin{theorem}\label{main1}
	Let Assumption \ref{Assump} be satisfied. Then there exists $c^*>0$ such that for any $c\ge c^*$, system \eqref{u-equation} has a traveling wave solution. Moreover, any traveling wave solution $w=w(a,\xi)$ is nonincreasing with respect to $\xi\in\R$. 
\end{theorem}

This theorem provides an existence result for equation \eqref{I-equation}. Indeed we have the following theorem
\begin{theorem}\label{main2}
	Let Assumption \ref{Assump} be satisfied and assume furthermore that $\mu: [0, a^+)\to\R_+$ and $\beta:[0, a^+)\to\R_+$ are continuous functions satisfying
	$$
	\int_0^{a^+}\mu(a)da=\infty \;\big(\Leftrightarrow \pi(a^+)=0\big), \quad \int_0^{a^+}\beta(a)\exp\left(-\int_0^a\mu(a')da'\right)da=1.
	$$
	Then there exists $c^*>0$ such that equation \eqref{I-equation} has a traveling wave solution for any wave speed $c\ge c^*$ and there is no traveling wave solution of equation \eqref{I-equation} if $c<c^*$. Moreover, the obtained traveling wave for $c\ge c^*$ vanishes at the maximum age, that is at $a=a^+$. 
\end{theorem}
The proofs of these results are based on the comparison principle to construct suitable super- and sub-solutions for equation \eqref{w-equation}, motivated by Ducrot \cite{ducrot2007travelling}. An important difference between this work and \cite{ducrot2007travelling} (dealing with with random diffusion) is that the solution map for \eqref{w-equation} has no compactness property, thus the Schauder fixed point theorem may not directly be applied. However, \eqref{w-equation} preserves the monotonicity and thus a monotone iteration method still works and allows us to overcome the lack of smoothing effect for nonlocal diffusion equation. In addition, we establish the required regularity of traveling wave solutions with respect to the variable $\xi$ in Definition \ref{TW} using suitable comparison arguments. Thus this regularity will make the term $\frac{\partial U}{\partial\xi}$, appearing in \eqref{U-equation}, to be well-defined at least almost everywhere.

Our second aim in this work is to study the spreading speeds of the solutions of the Cauchy problem, that is of the following initial value problem associated to \eqref{u-equation}
\begin{equation}\label{IVP-u}
	\begin{cases}
		\left(\frac{\partial u}{\partial t}+\frac{\partial u}{\partial a}\right)(t, a, x)=\left(J\ast u-u\right)(t, a, x)+\int_0^{a^+}K(a, a')\pi(a')u(t, a', x)da'\,(1-u(t, a, x)),\\
		u(t, 0, x)=\int_0^{a^+}\gamma(a)u(t, a, x)da,\\
		u(0, a, x)=u_0(a, x).
	\end{cases}
\end{equation} 
Then our result reads as follows.
\begin{theorem}\label{SS}
	Let Assumption \ref{Assump} be satisfied and assume that the initial data $u_0\in C([0, a^+]\times\R)$ with $u_0\le1$ is compactly supported, that is, ${\rm supp}(u_0)\subset[0, a^+]\times[-R, R]$ for some $R>0$. Recalling that $c^*$ is defined in Theorem \ref{main2}, the solution $u=u(t, a, x)$ of problem \eqref{IVP-u} enjoys the following properties
	\begin{itemize}
		\item [(i)] For all $c>c^*$, it holds that
		$$
		\lim\limits_{t\to\infty}\sup_{|x|\ge ct, 0<a<a^+}u(t, a, x)=0.
		$$
		
		\item [(ii)] For all $0\le c<c^*$, it holds that
		$$ 
		\lim\limits_{t\to\infty}\inf_{|x|\le ct, 0<a<a^+}u(t, a, x)=1.
		$$
	\end{itemize}
\end{theorem}
\begin{remark}
	{\rm The nonexistence of traveling wave solution of \eqref{original} for $c<c^*$ is a direct consequence of the above spreading property (Theorem \ref{SS}).
}
\end{remark}
Since the functions $\pi$ and $\ga$ do not depend on the spatial variable $x$, the above spreading speed results are established by comparing the solutions of \eqref{IVP-u} with whose of the nonlocal diffusion problem with constant demographic functions. More precisely, for outer spreading, namely $c>c^*$, we can use the classical super-solution based on the traveling wave profile, while for inner spreading, namely $c<c^*$, we construct sub-solutions with suitable wave speed using a Fisher-KPP equation with nonlocal diffusion. 
  
The paper is organized as follows. In Section \ref{P}, we discuss the well-posedness of \eqref{IVP-u} and construct suitable super-/sub-solutions for equation \eqref{w-equation}. In Section \ref{T},  we prove the existence of traveling wave solutions via monotone iteration method and derive regularity estimates for  traveling wave solutions. In Section \ref{S}, we study the spreading speed for some solutions of \eqref{IVP-u} and complete the proof of Theorem \ref{SS}.

\section{Preliminaries}\label{P}

\subsection{Well-posedness of \eqref{IVP-u}}\label{well}
In this section, we first study the well-posedness of \eqref{IVP-u}. To this aim, we introduce some functional framework and we define
 \begin{equation*}
 X=C_b(\R),\;\; \mathcal X=X\times L^1((0, a^+), X)\text{ and }\mathcal X_0=\{0_X\}\times L^1((0, a^+), X),
\end{equation*} 
wherein $0_X$ denotes the zero element in $X$. These spaces become Bamach spaces when they are endowed with the usual product norms and we also introduce their positive cones as follows
\begin{eqnarray}
	\mathcal{X}^+&=&X_+\times L^1_+((0, a^+), X)\nonumber\\
	&=&X_+\times\{u\in L^1((0, a^+), X): u(a, \cdot)\in X_+,\; \text{a.e. in }(0, a^+)\},\nonumber\\
	\mathcal{X}^+_0&=&\mathcal{X}^+\cap \mathcal{X}_0\text{ with } X_+:=\{u\in C_b(\R): u\ge0\}.\nonumber
\end{eqnarray}
We also define the linear positive and bounded convolution operator $T\in\mathcal L(X)$ by
\begin{equation}\label{op-K}
	[T\varphi](\cdot)=\int_\R J(\cdot-y)\varphi(y)dy,\;\forall \varphi\in X.
\end{equation}	
Note that due to Assumption \ref{Assump}-(iii) one has
\begin{equation}\label{norm-K}
	\|T\|_{\mathcal L(X)}\leq \int_{\R} J(y)dy:=1.
\end{equation}
Next define the operator $\mathcal B: dom(\mathcal B)\subset\mathcal X\to\mathcal X$ by
$$
dom(\mathcal B)=\{0_X\}\times W^{1, 1}((0, a^+), X)\text{ and }\mathcal B\begin{pmatrix} 0\\ \psi\end{pmatrix}=\begin{pmatrix}-\psi(0)\\  -\frac{\partial\psi}{\partial a}+[T-I]\psi\end{pmatrix},\;\;\forall \begin{pmatrix} 0\\ \psi\end{pmatrix}\in dom(\mathcal B).
$$
Note that $\mathcal B$ is a closed Hille-Yosida operator. Moreover, define the nonlinear operator $\mathcal C$ from $\mathcal X_0$ to $\mathcal X$ as follows:
$$
\mathcal C\begin{pmatrix} 0\\ \psi\end{pmatrix}=\begin{pmatrix}\int_{0}^{a^+}\ga(a)\psi(a)da\\  \int_0^{a^+}K(\cdot, a')\pi(a')\psi(a')da'\,(1-\psi)\end{pmatrix},\;\forall\begin{pmatrix} 0\\ \psi\end{pmatrix}\in \mathcal X_0.
$$
Then by identifying $U(t)=\begin{pmatrix}
	0\\ u(t)
\end{pmatrix}$, problem \eqref{IVP-u} rewrites as the following abstract Cauchy problem:
\begin{equation}\label{Cauchy}
	\begin{cases}
		\frac{dU}{dt}=\mathcal{B}U+\mathcal CU,\\
		U(0)=U_0,
	\end{cases}\text{ with }U_0=\begin{pmatrix}
		0\\ u_0
	\end{pmatrix}\in\mathcal{X}_0.
\end{equation}
Based on the Lipshcitz property of $\mathcal C$, by Thieme \cite{thieme1998positive,thieme2009spectral} or Magal and Ruan \cite{magal2018theory}, the existence and uniqueness of a mild solution for \eqref{Cauchy} is guaranteed. Next, recalling the definition of $\mathcal{B}$, one may observe that $\mathcal{B}$ is resolvent positive. Moreover, there exists some constant $L>0$ such that the operator $\mathcal C+LId_X$ is monotone on the subset $\mathcal S_0\subset \mathcal X_0$ given by
\begin{equation*}
\mathcal S_0:=\left\{\begin{pmatrix}0\\\varphi\end{pmatrix}\in \mathcal X_0,\;\;\varphi\leq 1\right\},
\end{equation*}
which means that for all $(U,V)\in \mathcal S_0$ one has
$$
U\le V\Rightarrow \mathcal CU+LU\le \mathcal CV+LV,
$$ 
where the partial order $\le$ in $\mathcal X_0$ and $\mathcal X$ is induced by the positive cones $\mathcal{X}_0^+$ and $\mathcal{X}^+$, respectively. 
To see this let us observe that for any $k>0$, if $L>k$ then the function $(v,u)\mapsto v(1-u)+Lu$ is increasing with respect to both variables $(v,u)\in (-\infty,k]\times (-\infty,1]$.\\
Due to the above monotonicity property, one also obtains that $\mathcal CU+LU\in \mathcal X^+$ for all $U\in \mathcal X_0^+\cap \mathcal S_0$. Hence the result by Magal et al. \cite[Theorem 4.5]{magal2019monotone} applies and ensures that 
Problem \eqref{Cauchy} is well posed in $\mathcal X_0^+\cap \mathcal S_0$, which is forward invariant and the comparison principle holds for \eqref{Cauchy} in $\mathcal S_0$. The latter result can be written as follows. 
\begin{lemma}[Comparison Principle]\label{wcp}
	Assume that $U_0, V_0\in\mathcal{S}_0$ and $U_0\leq V_0$, then the mild solutions $U(t)$ and $V(t)$ to \eqref{Cauchy} with initial data $U_0$ and $V_0$ respectively, satisfy $U(t)\leq V(t)$ for any $t\ge0$.
	
\end{lemma}
It follows that the comparison principle also holds for \eqref{IVP-u}.

\subsection{Limiting behavior}\label{LB}
In this sub-section, we consider the steady state equation \eqref{equi2} below. The study of the problem will be used in the sequel to study the limit behavior of traveling wave solutions. 
    	\begin{equation}\label{equi2}
    		\begin{cases}
    			\frac{d u(a)}{d a}=\int_0^{a^+}K(a, a')\pi(a')u(a')da'\,(1-u(a)),\; a\in [0, a^+],\\
    			u(0)=\int_0^{a^+}\ga(a)u(a)da.
    		\end{cases}
    	\end{equation}
Due to Assumption \ref{Assump} $(i)$, let us observe that \eqref{equi2} has constant solutions: $a\mapsto 0$ and $a\mapsto 1$. We will show that \eqref{equi2} only has these two solutions with values into $[0, 1]$. 
More specifically we have the following result.
\begin{lemma}\label{LE-lim}
Problem \eqref{equi2} only has the two solutions $u\in W^{1,1}(0,a^+)$ such that $0\leq u\leq 1$: $u\equiv 0$ and $u\equiv 1$.
\end{lemma}
\begin{proof}
To prove the above lemma we argue by contradiction by assuming that there exists a function $u^*\in W^{1, 1}(0, a^+)$ with $u^*\not\equiv 0, 1$, which is another solution of \eqref{equi2} satisfying $0\le u^*(a)\le 1$ for all $a\in [0,a^+]$. Observe that since $K$ is continuous and $u^*$ is continuous by Sobolev embedding, then $a\mapsto u^*(a)$ is of class $C^1$ on $[0,a^+]$.

First we claim that $0<u^*(a)<1$ holds for all $a\in [0, a^+]$. To prove this claim first note that if $u^*$ touches $1$ at some $a_0\in [0, a^+]$, then by the Cauchy Lipschitz theorem for ODE, one has $u^*\equiv 1$, which is a contradiction. Next if $u^*$ touches $0$ at some point $a_0\in [0, a^+]$, namely $u^*(a_0)=0$, then two cases may appear: either $a_0=0$ or $a_0\in (0, a^+]$.
If $a_0=0$ then the boundary condition at $a=0$ yields
$$
0=u(a_0)=\int_0^{a^+}\ga(a)u^*(a)da.
$$
But since $\gamma\not\equiv 0$ there exists $a_1\in (0,a^+)$ such that $u^*(a_1)=0$.
Therefore it is sufficient to consider the case where there exists $a_0\in (0,a^+]$ such that $u^*(a_0)=0$. Now to reach a contradiction, recalling $u^*\geq 0$, there holds $u(a_0)=0$ with $a_0\in (0,a^+]$ which ensures that
$$
\frac{du(a_0)}{da}\leq 0.
$$ 
Plugging this information together with $u^*(a_0)=0$ into \eqref{equi2} yields
	$$
	\int_0^{a^+}K(a_0, a')\pi(a')u^*(a')da'=0,
	$$
	which by the positivity of $K$ and $\pi$ implies that $u^*(a)\equiv0$, which is a contradiction again. Thus we have obtained that $0<u^*(a)<1$ for all $a\in [0,a^+]$.

Now to complete the proof of the lemma let us obtain a contradiction with the existence of $u^*$. To that aim let us show that one has $u^*(a)\geq 1$ for all $a\in [0,a^+]$ so that $u^*\equiv 1$, a contradiction with $u^*\not\equiv 1$. 
To reach this issue, let us define $\kappa^*$ by
$$ 
\kappa^*:=\inf\{\kappa>0: \kappa u^*(a)\ge 1,\,\,\forall a\in[0, a^+]\}. 
$$
Note that $\kappa^*$ is well-defined. Let us show that $\kappa^*\leq1$ by a contradiction argument and assuming that $\kappa^*>1$. Since $\kappa^*u^*(a)\ge 1$ for all $a\in [0,a^+]$, we consider two cases: (i) $\kappa^*u^*\ge 1$ and  $\kappa^*u^*\not\equiv 1$ and (ii) $\kappa^*u^*\equiv 1$.

\noindent {\bf Case (i)} $\kappa^*u^*\ge 1, \kappa^*u^*\not\equiv 1$. 
Set $v^*(a)=\kappa^* u^*(a)$ that satisfies the problem
\begin{equation}\label{equi2-bis}
    		\begin{cases}
    			\frac{d v^*(a)}{d a}=\int_0^{a^+}K(a, a')\pi(a')v^*(a')da'\,\left(1-\frac{1}{\kappa^*}v^*(a)\right),\; a\in [0, a^+],\\
    			v^*(0)=\int_0^{a^+}\ga(a)v^*(a)da.
    		\end{cases}
    	\end{equation}
Note that from the definition of $\kappa^*$, there exists $a_0\in [0,a^+]$ such that $v^*(a_0)=1$. Due to the condition
$$
\int_0^{a^+}\gamma(a)da=1,
$$
as above, the boundary condition ensures that one can choose $a_0\in (0,a^+]$.
Hence we have
$$
\frac{d v^*(a_0)}{d a}\leq 0,
$$
so that
$$
\int_0^{a^+}K(a_0, a')\pi(a')v^*(a')da'\,\left(1-\frac{1}{\kappa^*}\right)\leq 0.
$$
Since $\kappa^*>1$ and $K>0$, this yields $v^*\equiv 0$, a contradiction with $0<u^*<1$, which proves $\kappa^*\leq 1$ in the first case (i).\\

\noindent{\bf Case (ii)} $\kappa^*u^*=v^*\equiv1$. In that case \eqref{equi2-bis} yields for all $a\in (0,a^+)$ one has
$$
0=\int_0^{a^+}K(a, a')\pi(a')da'\,\left(1-\frac{1}{\kappa^*}\right),
$$
which is a contradiction again, since $\kappa^*>1$ and thus the right hand side is positive. Thus we still have $\kappa^*\leq1$. \\
To sum-up  we have proved that $u^*(a)\ge 1$ for all $a\in [0,a^+]$. We also have $1\ge u^*$. This proves that $1$ is the unique positive solution of \eqref{equi2} and complete the proof of the lemma.
\end{proof}

\subsection{Construction of sub-and super-solutions}\label{css}

In this section, we will construct exponentially bounded sub and super-solutions for problem \eqref{w-equation}. Exponential bound is used to define the convolution with the kernel $J$. As explained above, theses functions will be of particular importance to derive our existence result. 
\subsubsection*{Super-solution}
We start with the construction of a super-solution. We are looking for an exponentially bounded  function $\overline{U}=\overline{U}(a,\xi)>0$ satisfying
\begin{equation}\label{super-solution1}
\begin{cases}
\frac{\partial U(a, \xi)}{\partial a}\geq \left(J\ast_\xi U-U\right)(a, \xi)+c\frac{\partial U(a, \xi)}{\partial \xi}+(1-U(a,\xi))\int_0^{a^+}K(a, a')\pi(a')U(a', \xi)da',\\
U(0, \xi)\geq\int_0^{a^+}\ga(a)U(a, \xi)da,
\end{cases}
\end{equation}
for some speed $c>0$ to be specified latter. 

Since $\overline U>0$ it is sufficient to construct $\overline U$ as a solution of the following linear problem
\begin{equation}\label{super-solution}
\begin{cases}
\frac{\partial U(a, \xi)}{\partial a}\ge \left(J\ast_\xi U-U\right)(a, \xi)+c\frac{\partial U(a, \xi)}{\partial \xi}+\int_0^{a^+}K(a, a')\pi(a')U(a', \xi)da',\\
U(0, \xi)\ge\int_0^{a^+}\ga(a)U(a, \xi)da.
\end{cases}
\end{equation}
We now construct such a super-solution by looking for it with the following form
$$ 
\overline{U}(a, \xi)=e^{-\lambda \xi}\phi(a),
$$
with $\lambda>0$ and $\phi(a)>0$. Due to the differential inequality \eqref{super-solution}, we are looking for $\lambda>0$ and $\phi > 0$ such that
\begin{equation}\label{overline U}
\begin{cases}
\frac{\partial\overline{U}(a, \xi)}{\partial a}=\left(J\ast_\xi \overline{U}-\overline{U}\right)(a, \xi)+c\frac{\partial\overline{U}(a, \xi)}{\partial \xi}+\int_0^{a^+}K(a, a')\pi(a')\overline{U}(a', \xi)da',\\
\overline{U}(0, \xi)=\int_0^{a^+}\gamma(a)\overline{U}(a, \xi)da.
\end{cases}
\end{equation}
To solve this problem, we end-up with the following equation for $\la>0, c\in\R$ and $\phi=\phi(a)>0$.
\begin{equation}\label{phi}
	\begin{cases}
		\phi'(a)=\La(\la, c)\phi(a)+\int_0^{a^+}K(a, a')\pi(a')\phi(a')da',\\
		\phi(0)=\int_0^{a^+}\ga(a)\phi(a)da,
	\end{cases}
\end{equation}
wherein we have set
$$
\La(\la, c)=\int_\R J(y)e^{\la y}dy-1-c\la.
$$
Setting $s=\La(\la, c)$, the above problem is equivalent to find $s\in\R$ and $\phi=\phi(a)>0$ such that
\begin{equation}\label{phi_2}
\phi(a)=e^{sa}\int_0^{a^+}\ga(a)\phi(a)da+\int_0^ae^{s(a-l)}\int_0^{a^+}K(l, a')\pi(a')\phi(a')da'dl,\;\forall a\in [0,a^+].
\end{equation}
Define the linear operator $L_s\in \mathcal L(L^1(0, a^+))$ for any parameter $s\in\R$ and $\phi\in L^1(0, a^+)$ as follows,
\begin{equation}\label{L_s}
[L_s\phi](a):=e^{sa}\int_0^{a^+}\ga(a)\phi(a)da+\int_0^ae^{s(a-l)}\int_0^{a^+}K(l, a')\pi(a')\phi(a')da'dl.
\end{equation}
	We claim now that $L_s$ enjoys the following properties:
	\begin{itemize}
		\item [(1)] $L_s$ is a positive operator, that is $L_s\left(L^1_+(0,a^+)\right)\subset L^1_+(0,a^+)$;
		
		\item [(2)] $L_s$ is a compact operator;
		
		\item [(3)] $L_s$ is a non-supporting operator.
	\end{itemize}
    Note that (1) is true due to Assumption \ref{Assump}-(i), (ii) and (iv). Next, applying Fr{\'e}chet-Kolmogorov-Riesz compactness theorem to $L_s$ yields (2). To prove (3), by taking the dual product between $L_s\phi$ with $\phi\in L^1_+(0, a^+)\setminus\{0\}$ and any $\psi\in L^\infty_+(0, a^+)\setminus\{0\}$, we obtain 
    $$
    \langle \psi, L_s\phi\rangle=\int_0^{a^+}\ga(a)\phi(a)da\int_0^{a^+}\psi(a)e^{sa}da+\int_0^{a^+}\psi(a)\int_0^ae^{s(a-l)}\int_0^{a^+}K(l, a')\pi(a')\phi(a')da'dlda.
    $$
    Recalling Assumption \ref{Assump}-(ii) and (iv), one can see that the second term of the above equality is positive. Thus (3) is proved.
    
    Hence by Krein-Rutman theorem, the spectral radius $\rho(L_s)$ is the principal eigenvalue of $L_s$ (see Sawashima \cite{sawashima1964spectral}) and moreover the map $s\to \rho(L_s)$ is continuous and strictly increasing (see Marek \cite{marek1970frobenius}). In addition, notice that $\rho(L_s)\to0$ as $s\to-\infty$.
    
    On the other hand, for $s=0$, applying the constant function $a\mapsto1$ to $L_0$ yields, due to Assumption \ref{Assump}-(i),
    $$
    [L_01](a)=1+\int_0^a\int_0^{a^+}K(l, a')\pi(a')da'dl>1,\;\forall a\in [0,a^+],
    $$
    so that $\rho(L_0)>1$. It follows that there exists a unique $s_0<0$ such that $\rho(L_{s_0})=1$. We now consider the equation $\La(\la, c)=s_0$, that reads as
\begin{equation}\label{s_0}
	\int_\R J(y)e^{\la y}dy-1-c\la=s_0.
\end{equation}
Thus we have the following lemma.
\begin{lemma}\label{speed}
	There exists $c^*>0$ such that equation \eqref{s_0} with respect to the unknown $\la$ has no positive solution when $c<c^*$, only one solution for $c=c^*$, and exactly two solutions for $c>c^*$. In addition for $c>c^*$, these two solutions $\lambda_1$ and $\lambda_2$ are such that
	$$ 
	0< \lambda_1<\lambda(c)<\lambda_2,
	$$
	where $\lambda(c)$ is the unique root of the following equation
	$$
	\int_\R J(y)e^{\lambda(c) y}ydy=c.
	$$
\end{lemma}
We will see in the proof of the lemma that the critical value $c^*$ is implicitly given by the equation
\begin{equation}\label{c*}
	\int_\R J(y)e^{\la(c^*)y}dy-1-c^*\la(c^*)=s_0,
\end{equation}
which uniquely provides the value of $c^*>0$.

\begin{proof}
First, observe that 
$$
\frac{\partial\La(\la, c)}{\partial\la}=\int_\R J(y)e^{\la y}ydy-c,\;\;\frac{\partial^2\La(\la, c)}{\partial\la^2}=\int_\R J(y)e^{\la y}y^2dy>0,
$$ 
and
$$
\int_\R J(y)e^{\la y}ydy=\int_0^\infty yJ(y)\left[e^{\la y}-e^{-\la y}\right]dy\to \pm\infty\text{ as }\lambda\to\pm\infty.
$$
As a consequence, for any $c>0$, the function $\la\mapsto \frac{\partial\La(\la, c)}{\partial\la}$ is non-decreasing from $-\infty$ to $\infty$ and, using again the symmetry of $J$, satisfies 
$$
\frac{\partial\La(0, c)}{\partial\la}=-c<0.
$$
Hence for any $c>0$ there exists a unique $\lambda(c)>0$ such that $\int_\R J(y)e^{\lambda y}ydy-c<0$ (resp. $>0$) for all $\lambda<\lambda(c)$ (resp. for all $\lambda>\lambda(c)$).
We also deduce that for any $c>0$ the function $\lambda\to \La(\lambda, c)$ is decreasing on $(0, \lambda(c))$ and increasing on $(\lambda(c), \infty)$.

As a consequence of the above analysis, the solution of \eqref{s_0} relies on the position of its minimum over $[0,\infty)$, i.e. $\La(\la(c), c)$, with respect to $s_0$. This quantity is given by the formula
$$
\La(\lambda(c), c)=\int_\R J(y)e^{\la(c)y}dy-1-c\la(c).
$$
To study the equation $\La(\lambda(c), c)=s_0$, let us note that the implicit function theorem ensures that the map $c\mapsto \lambda(c)$ is of class $C^1$ with respect to $c\in (0,\infty)$, so is the function $c\mapsto \La(\lambda(c), c)$ and we have:
\begin{eqnarray}
	\frac{d \La(\lambda(c), c)}{d c}=\frac{\partial \La(\lambda(c), c)}{\partial \lambda}\frac{\partial\lambda(c)}{\partial c}+\frac{\partial\La(\lambda(c), c)}{\partial c}=-\lambda(c)<0.\nonumber
\end{eqnarray}
The above expression implies that $c\mapsto \La(\lambda(c), c)$ is monotonically decreasing with respect to $c$. Furthermore, this function takes the value $0$ for $c=0$ and tends to $-\infty$ as $c\to\infty$. Indeed since $J$ is symmetric and has a unit mass, we have
\begin{eqnarray}
	&&\int_\R J(y)e^{\lambda(c)y}dy-c\lambda(c)\nonumber\\
	&=&\int_0^\infty J(y)e^{\lambda(c)y}\left[1-\lambda(c)y\right]dy+\int_0^\infty J(y)e^{-\lambda(c)y}\left[1+\lambda(c)y\right]dy\to-\infty,\text{ as }c\to\infty.\nonumber 
		\end{eqnarray}
Recalling that $s_0<0$, this proves there exists a unique $c^*>0$ such that
$$
\La(\lambda(c), c)\begin{cases} <s_0,\;\forall c\in (c^*,\infty),\\
>s_0,\;\forall c\in (0,c^*),
\end{cases}
$$
that concludes the proof of the lemma.
\end{proof}

To sum-up for all $c>c^*$, the function $\overline{U}(a,\xi)=e^{-\lambda\xi}\phi(a)$, where $\lambda$ is a solution of \eqref{s_0}, is a super-solution, that satisfies \eqref{super-solution1}.
Note also that the constant function $(a,\xi)\mapsto 1$ is also a super-solution.

\subsubsection*{Sub-solution}
In this subsection, we construct a suitable sub-solution. Let $c>c^*$ be fixed and let $\la>0$ be the smallest solution of \eqref{s_0}. Moreover, let $\phi=\phi(a)$ be defined as in \eqref{phi_2} with $s=s_0$. Observe that the function $a\mapsto\phi(a)$ is continuous and $\phi(a)>0$ for any $a\in[0, a^+]$.

Now our aim is to construct an exponentially bounded sub-solution with the speed $c$, that is a solution
$\underline{U}=\underline{U}(a,\xi)$ such that
$$
\underline{U}\leq \min\left\{1, \overline{U}\right\},
$$
and satisfying the following differential inequality
 \begin{equation}\label{sub-solution1}
\begin{cases}
\frac{\partial \underline{U}(a, \xi)}{\partial a}\leq \left(J\ast_\xi \underline{U}-\underline{U}\right)(a, \xi)+c\frac{\partial \underline{U}(a, \xi)}{\partial \xi}+(1-\underline{U}(a,\xi))\int_0^{a^+}K(a, a')\pi(a')\underline{U}(a', \xi)da',\\
\underline{U}(0, \xi)\leq\int_0^{a^+}\ga(a)\underline{U}(a, \xi)da.
\end{cases}
\end{equation}
Note that we have
$$
1-\underline{U}\geq\max\{0, 1-\overline{U}\}\ge\left(1-\alpha e^{-\lambda \xi}\right)^+,
$$
for some sufficiently large constant $\alpha>0$ and where the superscript $+$ denotes the positive part.\\
Hence it is sufficient to look for an exponentially bounded function $\underline{U}$ satisfying
\begin{equation}\label{sub}
\begin{cases}
\frac{\partial\underline{U}(a,\xi)}{\partial a}\le J\ast_\xi \underline{U}(a,\xi)-\underline{U}(a,\xi)+c\frac{\partial\underline{U}(a,\xi)}{\partial \xi}+\int_0^{a^+}K(a, a')\pi(a')\underline{U}(a',\xi)da'\,(1-\alpha e^{-\lambda \xi})^+,\\
\underline{U}(0, \xi)\leq\int_0^{a^+}\ga(a)\underline{U}(a, \xi)da.
\end{cases}
\end{equation}
We look for such a function $\underline{U}$ of the form
\begin{equation}\label{sub2}
\underline{U}(a, \xi)=(e^{-\lambda \xi}-ke^{-(\lambda+\eta)\xi})\phi(a),
\end{equation}
where the function $\phi$ is defined in \eqref{phi_2} with $s=s_0$, while $k>0$ and $\eta>0$ have to be determined.
We split our analysis into two regions: $1-\alpha e^{-\lambda \xi}>0$ and $\left(1-\alpha e^{-\lambda \xi}\right)^+=0$.
We define $\xi_M$ by
$$
\xi_M:=\frac{1}{\lambda}\ln \alpha,
$$
so that the two above regions becomes $\{\xi>\xi_M\}$ and $\{\xi\leq \xi_M\}$.\\
\noindent{\bf Case 1 :} $\xi>\xi_M$.\\
First simple computations show that $\underline{U}$ satisfies the differential inequality \eqref{sub} for $\xi>\xi_M$ if and only if we have for all $\xi>\frac{1}{\la}\log\alpha$ and all $a\in(0, a^+)$,
\begin{eqnarray}\label{k}
&-k\phi(a)&\left(\int_\R J(y)e^{\lambda y}dy-c\lambda-\int_\R J(y)e^{(\lambda+\eta)y}dy+c(\lambda+\eta)\right)\nonumber\\
&&\leq\alpha\int_0^{a^+}K(a, a')\pi(a')\phi(a')da'\, e^{-\lambda \xi}(k-e^{\eta \xi}).
\end{eqnarray}
Then if we set
$$
p(\eta):=\int_\R J(y)e^{\lambda y}dy-c\lambda-\int_\R J(y)e^{(\lambda+\eta)y}dy+c(\lambda+\eta),
$$
the above inequality becomes for all $\xi>\frac{1}{\la}\log\alpha$ and all $a\in(0, a^+)$,
\begin{equation*}
-k\phi(a)p(\eta)\leq\alpha\int_0^{a^+}K(a, a')\pi(a')\phi(a')da'\, e^{-\lambda \xi}(k-e^{\eta \xi}).
\end{equation*}
Recalling that $\lambda<\lambda(c)$, function satisfies $p(0)=0$ and $p'(0)>0$. Therefore for any sufficiently small and nonnegative $\eta$, we have $p(\eta)>0$. Next, for $k$ sufficiently large and $\eta$ small enough, due to $\inf_{a\in[0, a^+]}\phi(a)>0$ it is sufficient to have for any $\xi>\xi_M$:
$$
-kp(\eta)\phi(a)+\alpha\int_0^{a^+}K(a, a')\pi(a')\phi(a')da'\, e^{(-\lambda+\eta)\xi}<0, \;\forall a\in[0, a^+].
$$
Therefore, with such a choice of parameters $k$ and $\eta$, inequality \eqref{k} is satisfied for any $\xi>\xi_M$, which completes Case 1.\\
\noindent{\bf Case 2 :} $\xi\leq \xi_M$.\\
For $\xi\leq \xi_M$, $\underline{U}$ satisfies \eqref{sub} if and only if the following inequality holds: 
$$
\int_0^{a^+}K(a, a')\pi(a')\phi(a')da'\,(1-ke^{-\eta \xi})\le k\phi(a)e^{-\eta \xi}p(\eta), \; \text{ for all $a\in[0, a^+]$ and $\xi\leq \xi_M$}.
$$
In order to reach this inequality, it is sufficient to have
$$
\left(\sup_{a\in [0,a^+]}\int_0^{a^+}K(a, a')\pi(a')\phi(a')da'\right) e^{\eta \xi_M}\le kp(\eta)\inf_{a\in [0,a^+]}\phi(a),
$$
which is satisfied for sufficiently large values of $k$, since $p(\eta)>0$. This completes the Case 2.\\
In summary, we have proved that $\underline U$ satisfies \eqref{sub} for all $\xi\in\R$ as soon as $\eta>0$ is small enough and $K>0$ is chosen large enough.

Next, define the functions 
$$
\underline{w}(a, \xi)=\underline{U}(a, \xi-ca),\; \overline{w}(a, \xi)=\overline{U}(a, \xi-ca),
$$ 
so that $\underline{w}\leq \overline{w}$.
These functions will be used in the next section for the construction of wave solutions and to derive our existence result.

\section{Traveling wave solution}\label{T}
In this section, we prove Theorems \ref{main1} and \ref{main2}.
\subsection{Existence}\label{existence}
In this subsection, we will employ a monotone iteration method to derive an existence result for traveling wave solutions. To this aim, we study the following nonlinear problem with $c\in\R$.
\begin{equation}\label{nonlinear}
\begin{cases}
\frac{\partial u}{\partial a}(a, \xi)=(J\ast_\xi u-u)(a, \xi)+\int_0^{a^+}K(a, a')\pi(a')u(a', \xi)da'\, (1-u(a, \xi)),&(a, \xi)\in(0, a^+)\times\R,\\
u(0, \xi)=\int_0^{a^+}\gamma(a)u(a, \xi+ca)da, &\xi\in\R.
\end{cases}
\end{equation}
\begin{definition}
	We say that a function $\overline{u}\in W^{1, 1}((0, a^+), C(\R))$ satisfying 
	\begin{equation}\label{exp}
    \sup_{a\in(0, a^+)}|\overline u(a, \xi)|=\mathcal O(e^{\tau|\xi|}), \text{ as }\xi\to\pm\infty, \text{ for some }\tau>0
    \end{equation}
    is a super-solution of \eqref{nonlinear}, if it satisfies that
	\begin{equation}\label{nonlinear2}
	\begin{cases}
	\frac{\partial \overline{u}}{\partial a}(a, \xi)\ge (J\ast_\xi \overline{u}-\overline{u})(a, \xi)+\int_0^{a^+}K(a, a')\pi(a')\overline u(a', \xi)da'\, (1-\overline u(a, \xi)),&(a, \xi)\in [0, a^+]\times\R,\\
	\overline{u}(0, \xi)\ge \int_0^{a^+}\gamma(a)\overline{u}(a, \xi+ca)da, &\xi\in\R.
	\end{cases}
	\end{equation}
	Similarly, we define a sub-solution by interchanging the inequalities. Here, \eqref{exp} is used to guarantee the convolution $J\ast \overline u$ to be well-defined.
\end{definition}
\begin{proposition}\label{sub-super method}
	Assume that there exists a pair of sub/super-solution of \eqref{nonlinear} such that $0\le\underline{u}\le\overline{u}\le 1$, then there exist a minimal solution $u_*$ and a maximal solution $u^*$ of \eqref{nonlinear}, in the sense that for any other solution $u\in \{v\in W^{1, 1}((0, a^+), L^1_{loc}(\R)): \underline{u}\le v\le \overline{u}\}$, it holds that
	$$
	\underline{u}\le u_*\le u\le u^*\le \overline{u}.
	$$
\end{proposition}
\begin{proof}
	Define $M>0$ as follows,
	\begin{equation}\label{M}
	M=\max_{a\in[0, a^+]}\int_0^{a^+}K(a, a')\pi(a')da',
	\end{equation}
	and define the sequence of the functions $(u_n)_{n\geq 0}$ as $u_0=\underline{u}$ and for $n\ge1$, for $(a,\xi)\in [0,a^+]\times \R$ by
	\begin{equation}\label{u_n}
	\begin{cases}
	\frac{\partial u_n(a, \xi)}{\partial a}- (J\ast_\xi u_n-u_n-Mu_n)(a, \xi)=\int_0^{a^+}K(a, a')\pi(a')u_{n-1}(a', \xi)da'\, (1-u_{n-1}(a, \xi))+Mu_{n-1}(a, \xi),\\
	u_n(0, \xi)= \int_0^{a^+}\gamma(a)u_{n-1}(a, \xi+ca)da.
	\end{cases}
	\end{equation}
	Next, set $u^0=\overline{u}$ and define $u^n$ with $n\ge1$, for $(a,\xi)\in [0,a^+]\times \R$, by the resolution of the problem
	\begin{equation}\label{u^n}
	\begin{cases}
	\frac{\partial u^n(a, \xi)}{\partial a}- (J\ast_\xi u^n-u^n-Mu^n)(a, \xi)=\int_0^{a^+}K(a, a')\pi(a')u^{n-1}(a', \xi)da'\, (1-u^{n-1}(a, \xi))+Mu^{n-1}(a, \xi),\\
	u^n(0, \xi)= \int_0^{a^+}\gamma(a)u^{n-1}(a, \xi+ca)da.
	\end{cases}
	\end{equation}
	First we show that $u_n\in W^{1, 1}((0, a^+), C(\R))$ is well defined. Note that
	$$
	\int_0^{a^+}K(a, a')\pi(a')u_0(a', \xi)da'\, (1-u_0(a, \xi))=\int_0^{a^+}K(a, a')\pi(a')\underline u(a', \xi)da'\, (1-\underline u(a, \xi)), 
	$$
	so that $u_1$ is well defined by the linear nonlocal diffusion equations. Analogously by induction we get that $u_n$ is well defined. Similarly, the existence of $u^n$ follows the same induction arguments.
	
	We will show that the sequence $u_n$ (respectively $u^n$) is increasing (respectively decreasing) in the sense that
	\begin{equation}\label{inequality}
	\underline{u}\le\cdots\le u_n\le u_{n+1}\le u^{n+1}\le u^n\le\cdots\le \overline{u}.
	\end{equation}
	Indeed, taking $w:=u_1-u_0$, it satisfies
	\begin{equation*}
	\begin{cases}
	\frac{\partial w(a, \xi)}{\partial a}- (J\ast_\xi w+w+Mw)(a, \xi) \ge 0
	,\\
	w(0, \xi)\ge 0.
	\end{cases}
	\end{equation*}
	Using the comparison principle of nonlocal diffusion problems, we conclude that $w\ge0$, i.e. $\underline{u}=u_0\le u_1$. Now assume that $u_{n-1}\le u_n$. Observe that
    \begin{eqnarray}
    	&&\int_0^{a^+}K(a, a')\pi(a')u_n(a', \xi)da'\,(1-u_n(a, \xi))+Mu_n(a, \xi)\nonumber\\
    	&&-\int_0^{a^+}K(a, a')\pi(a')u_{n-1}(a', \xi)da'\,(1-u_{n-1}(a, \xi))-Mu_{n-1}(a, \xi)\nonumber\\
    	&\ge&\int_0^{a^+}K(a, a')\pi(a')u_{n-1}(a', \xi)da'(u_{n-1}-u_n)(a, \xi)\,+M(u_n-u_{n-1})(a, \xi)\nonumber\\
    	&\ge&0.\nonumber
    \end{eqnarray}
	Then $w:=u_{n+1}-u_n$ satisfies
	\begin{equation*}
	\begin{cases}
	\frac{\partial w(a, \xi)}{\partial a}- (J\ast_\xi w+w+Mw)(a, \xi)\ge 0,\\
	w(0, \xi)\ge\int_0^{a^+}\gamma(a)(u_n-u_{n-1})(a, \xi+ca)da\ge0.
	\end{cases}
	\end{equation*}
	Again the comparison principle shows that $u_n\le u_{n+1}$. The rest of the inequalities in \eqref{inequality} can be proved similarly.
	
	Next, for all $(a, \xi)\in(0, a^+)\times\R$, $u_n(a, \xi)$ has a limit as $n\to\infty$ by monotone convergence theorem, denoted by $u_*(a, \xi)$; that is $u_n(a, \xi)\to u_*(a, \xi)$ for all $(0, a^+)\times\R$. Thus, one has that for all $\xi\in\R$,
	\begin{eqnarray}\label{unn}
		\int_{0}^{a^+}\ga(a)u_n(a, \xi+ca)da\xrightarrow{n\to\infty}
		\int_{0}^{a^+}\ga(a)u_*(a, \xi+ca)da.
	\end{eqnarray}
	In addition, one has 
	\begin{eqnarray}
		\varphi_n(a, \xi)&:=&J\ast_\xi u_n(a, \xi)-u_n(a, \xi)+\int_0^{a^+}K(a, a')\pi(a')u_n(a', \xi)da'\, (1-u_n(a, \xi))\nonumber\\
		&\xrightarrow{n\to\infty}& J\ast_\xi u_*(a, \xi)-u_*(a, \xi)+\int_0^{a^+}K(a, a')\pi(a')u_*(a', \xi)da'\, (1-u_*(a, \xi))\nonumber\\
		&:=&\varphi(a, \xi),\nonumber
	\end{eqnarray}
    in $(0, a^+)\times\R$ and in $L^1((0, a^+), L^1_{loc}(\R))$. Hence, for all $\xi\in\R$ and $(\eta, \theta)\subset(0, a^+)$, one has
	$$
	u_n(\theta, \xi)-u_n(\eta, \xi)=\int_\eta^\theta \varphi_n(a, \xi)da, \;\forall n\ge0,
	$$
	which implies by letting $n\to\infty$ that
	$$	
	u_*(\theta, \xi)-u_*(\eta, \xi)=\int_\eta^\theta \varphi(a, \xi)da.
	$$
	Since $\varphi\in L^1((0, a^+), L^1_{loc}(\R))$, it follows that $u_*\in W^{1, 1}((0, a^+), L^1_{loc}(\R))$ satisfies the first equation of \eqref{nonlinear} with $\partial_au_*=\varphi$ a.e. in $(0, a^+)\times\R$. Further, $u_*(\cdot, \xi)$ is continuous in $[0, a^+]$ for all $\xi\in\R$. Thus, one has $u_n(0, \xi)\to u_*(0, \xi)$ as $n\to\infty$ in $\R$, which by \eqref{unn} implies that
	$$
	u_*(0, \xi)=\int_{0}^{a^+}\ga(a)u_*(a, \xi+ca)da, \text{ a.e. }\xi\in\R.
	$$
	It follows that $u_*\in W^{1, 1}((0, a^+), L^1_{loc}(\R))$ satisfies the equation \eqref{nonlinear}. 
	
	We now claim that $u_*$ is the minimal solution of \eqref{nonlinear}. Indeed, if $u$ is a solution of \eqref{nonlinear} such that $u\in[\underline{u}, \overline{u}]$, it can be shown that the sequence $u_n$ built in \eqref {u_n} satisfies that $\underline{u}\le u_n\le u$. Hence, $u_n\uparrow u_*\le u$. We can argue similarly with the sequence $u^n$ and conclude $u^*$ is maximal. This completes the proof.
\end{proof}

Now define 
\begin{equation}\label{super}
\underline{u}(a, \xi)=\max\{\underline{w}(a, \xi),\, 0\}\text{ and }\overline{u}(a, \xi)=\min\{1,\, \overline{w}(a, \xi)\}.
\end{equation}
By the constructions proposed in Section \ref{css}, $\underline u$ and $\overline u$ are indeed the sub- and super-solutions of \eqref{nonlinear} respectively. Next, due to $u^0=\overline u\le1$ and recalling $M$ defined in \eqref{M}, 
Proposition \ref{sub-super method} applies to the system \eqref{w-equation} and one concludes to the existence of a maximal solution $w\in W^{1, 1}((0, a^+), L^1_{loc}(\R))$ satisfying
$$
\max\{\underline{w}(a, \xi),\, 0\}\le w(a, \xi)\le\min\{1,\, \overline{w}(a, \xi)\}.
$$
It follows that $w(a, \xi)\to0$ as $\xi\to\infty$ uniformly in $a\in[0, a^+]$. Hence to complete the proof of Theorem \ref{main1}, it remains to study the behavior of $w$ as $\xi\to-\infty$. Note that the function $\xi\to\overline w(a, \xi)$ is nonincreasing on $\R$, then the limit function $\xi\to w(a, \xi)$ is also nonincreasing on $\R$ for a.e. $a\in(0, a^+)$. 

On the other hand, recalling the definition of $\underline w$ in \eqref{sub2}, there exists $\xi_0$ large enough such that 
$$
\inf_{a\in(0, a^+)}\underline w(a, \xi_0)>0.
$$ 
Hence due to the monotonicity of $w$ with respect to $\xi$, it follows that 
$$
\lim_{\xi\to-\infty}\inf_{a\in(0, a^+)}w(a, \xi)>0.
$$
Moreover the monotonicity of $\xi\mapsto w(a, \xi)$, for any $a\in[0, a^+]$, ensures that the limit $w^-(a):=\lim\limits_{\xi\to-\infty}w(a, \xi)$ does exist for any $a\in [0,a^+]$. Hence, since $0\leq w(a,\xi)\leq 1$ for all $(a, \xi)\in[0, a^+]\times\R$, the Lebesgue convergence theorem yields
$$
\lim\limits_{\xi\to-\infty}\int_\R J(y)[w(a, \xi-y)-w(a, \xi)]dy=\int_\R J(y)[w^-(a)-w^-(a)]dy=0,\; \forall a\in [0,a^+].
$$ 
As a consequence $w^-$ satisfies \eqref{equi2}. Now recall that \eqref{equi2} has a unique positive solution, the constant function $a\mapsto 1$, which is shown in Section \ref{LB}, thus we have
$$
\lim\limits_{\xi\to-\infty}w(a, \xi)=1, \text{ for all $a\in [0, a^+]$},
$$
the convergence being in fact uniform for $a\in [0,a^+]$ since $\partial_a w(a,\xi)$ is globally bounded.

\subsection{Regularity}
At last, we show the regularity of $w$ with respect to $\xi$. To this aim, let us fix $\la>0$ as the smallest solution of \eqref{s_0}.
We now summarize some Lipschitz regularity estimates in the following lemma, motivated by Shen and Shen \cite{shen2015transition} and Ducrot and Jin \cite{ducrot2021generalized}. 
\begin{lemma}\label{Lip}
	There exists some constant $m\ge\la$ large enough such that for all $n\ge1$, one has
	$$
	|u^n(a, \xi+h)-u^n(a, \xi)|\le e^{m|h|}-1, \;\forall (a, \xi)\in[0, a^+]\times\R.
	$$
\end{lemma}
\begin{proof}
	In the following we prove it for $h>0$. The case for $h<0$ can be proved similarly.
	
    Recall that $u^0(a, \xi)=\overline{u}(a, \xi)$ defined in \eqref{super} is given as follows:
    $$
    u^0(a, \xi)=\begin{cases}
	e^{-\la(\xi-ca)}\phi(a), &\text{ if }\xi\ge ca+\frac{1}{\la}\ln\phi(a),\\
	1, &\text{ if }\xi<ca+\frac{1}{\la}\ln\phi(a),
	\end{cases}
$$
while for $h>0$, one has
$$
u^0(a, \xi+h)=\begin{cases}
	e^{-\la(\xi+h-ca)}\phi(a), &\text{ if }\xi+h\ge ca+\frac{1}{\la}\ln\phi(a),\\
	1, &\text{ if }\xi+h<ca+\frac{1}{\la}\ln\phi(a).
\end{cases}
$$
Then we infer from these formulas,
\begin{eqnarray}
	&&\frac{u^0(a, \xi+h)}{u^0(a, \xi)}\nonumber\\
	&\!\!=\!\!\!\!&\!\!\begin{cases}
		e^{-\la h}, \!\!\!\!&\text{ if }\xi\ge ca+\frac{1}{\la}\ln\phi(a),\\
		e^{-\la(\xi+h-ca)}\phi(a), \!\!\!\!&\text{ if }ca+\frac{1}{\la}\ln\phi(a)-h\le \xi< ca+\frac{1}{\la}\ln\phi(a),\\
		1, \!\!\!\!&\text{ if }\xi+h<ca+\frac{1}{\la}\ln\phi(a).
	\end{cases}\nonumber
\end{eqnarray}
Hence one can choose $m>\la$ large enough such that
$$
e^{-m h}\le\frac{u^0(a, \xi+h)}{u^0(a, \xi)}\le1, \;\forall (a, \xi)\in[0, a^+]\times\R.
$$
Now consider the functions $v_{n, h}=v_{n, h}(a, \xi), n\ge0$ given by
\begin{equation}\label{v_n}
v_{n, h}(a, \xi):=\frac{u^n(a, \xi+h)}{e^{-m h}}.
\end{equation}
Next, we choose $M>0$ a little bit different from \eqref{M}, since now $v_{n, h}$ could be larger than $1$. Observe that $v_{n, h}\le e^{mh}$ for all $n\ge0$, so set 
$$
M:=e^{mh}\,\max_{a\in[0, a^+]}\int_0^{a^+}K(a, a')\pi(a')da'.
$$
Direct computations for any $0\le v\le 1$ and $v\le u\le e^{mh}$ yields
\begin{eqnarray}
	&&\int_0^{a^+}K(a, a')\pi(a')u(a', \xi)da'\,(1-u(a, \xi))+Mu(a, \xi)\nonumber\\
	&&-\int_0^{a^+}K(a, a')\pi(a')v(a', \xi)da'\,(1-v(a, \xi))-Mv(a, \xi)\nonumber\\
	&=&\int_0^{a^+}K(a, a')\pi(a')(u(a', \xi)-v(a', \xi))da'(1-v(a, \xi))\nonumber\\
	&&-\int_0^{a^+}K(a, a')\pi(a')(u(a', \xi)-v(a', \xi))da'(u(a, \xi)-v(a, \xi))\nonumber\\
	&&-\int_0^{a^+}K(a, a')\pi(a')v(a', \xi)da'(u(a, \xi)-v(a, \xi))+M(u-v)(a, \xi)\nonumber\\
	&\ge&\left(M-\int_0^{a^+}K(a, a')\pi(a')u(a', \xi)da'\right)(u-v)(a, \xi)\nonumber\\
	&\ge&0.\label{g}
\end{eqnarray}
First let us consider $n=1$. Observe that $v_{1, h}$ satisfies the following equation,
\begin{equation*}
	\begin{cases}
		\frac{\partial v_{1, h}}{\partial a}=J\ast_\xi v_{1, h}-v_{1, h}-Mv_{1, h}+Mv_{0, h}+\int_0^{a^+}K(a, a')\pi(a')v_{0, h}(a', \xi)da'\,(1-u^{0}(a, \xi+h)),\\
		v_{1, h}(0, \xi)=\int_0^{a^+}\ga(a)v_{0, h}(a, \xi+ca)da.
	\end{cases}
\end{equation*}
It follows from $v_{0, h}(a, \xi)\ge u^0(a, \xi+h)$ that
\begin{equation*}
	\begin{cases}
		\frac{\partial v_{1, h}}{\partial a}\ge J\ast_\xi v_{1, h}-v_{1, h}-Mv_{1, h}+Mv_{0, h}+\int_0^{a^+}K(a, a')\pi(a')v_{0, h}(a', \xi)da'\,(1-v_{0, h}),\\
		v_{1, h}(0, \xi)=\int_0^{a^+}\ga(a)v_{0, h}(a, \xi+ca)da.
	\end{cases}
\end{equation*}
On the other hand, recall that $u^1$ satisfies \eqref{u^n}, thus due to the monotonicity \eqref{g} and $v_{0, h}\ge u^0$, we have 
\begin{eqnarray}
	&&Mv_{0, h}(a, \xi)+\int_0^{a^+}K(a, a')\pi(a')v_{0, h}(a', \xi)da'\,(1-v_{0, h}(a, \xi))\nonumber\\
	&\ge& Mu^{0}(a, \xi)+\int_0^{a^+}K(a, a')\pi(a')u^{0}(a', \xi)da'\, (1-u^{0}(a, \xi)),\nonumber
\end{eqnarray}
and the initial data satisfies, 
\begin{eqnarray}
v_{1, h}(0, \xi)&=&\int_0^{a^+}\ga(a)v_{0, h}(a, \xi+ca)da=e^{m h}\int_0^{a^+}\ga(a)u^0(a, \xi+ca+h)da\nonumber\\
&\ge& \int_0^{a^+}\ga(a)u^0(a, \xi+ca)da=u^1(0, \xi), \;\forall\xi\in\R.\nonumber
\end{eqnarray}
The comparison principle applies and provides
\begin{equation}\label{ge}
u^1(a, \xi)\le v_{1, h}(a, \xi), \;\forall (a, \xi)\in[0, a^+]\times\R.
\end{equation}
Now since for all $a\in[0, a^+]$ the function $\xi\to u^1(a, \xi)$ is nonincreasing, for all $a\in[0, a^+]$ and $\xi\in\R$, we get
\begin{eqnarray}\label{u0}
	|u^1(a, \xi+h)-u^1(a, \xi)|\le u^1(a, \xi)-u^1(a, \xi+h)&\le& (1-e^{-m h})v_{1, h}(
	a, \xi)\nonumber\\
	&\le& e^{m h}-1.
\end{eqnarray}
Next, let us prove the result \eqref{ge} for any $n>1$ by induction. Assume that $v_{n, h}(a, \xi)\ge u^n(a, \xi)$ for any $(a, \xi)\in[0, a^+]\times\R$. Observe that $v_{n+1, h}$ satisfies the following equation,
\begin{equation*}
	\begin{cases}
		\frac{\partial v_{n+1, h}}{\partial a}=J\ast_\xi v_{n+1, h}-v_{n+1, h}-Mv_{n+1, h}+Mv_{n, h}+\int_0^{a^+}K(a, a')\pi(a')v_{n, h}(a', \xi)da'\,(1-u^{n}(a, \xi+h)),\\
		v_{n+1, h}(0, \xi)=\int_0^{a^+}\ga(a)v_{n, h}(a, \xi+ca)da.
	\end{cases}
\end{equation*}
It follows from $v_{n, h}(a, \xi)\ge u^n(a, \xi+h)$ by \eqref{v_n} that
\begin{equation*}
	\begin{cases}
		\frac{\partial v_{n+1, h}}{\partial a}\ge J\ast_\xi v_{n+1, h}-v_{n+1, h}-Mv_{n+1, h}+Mv_{n, h}+\int_0^{a^+}K(\cdot, a')\pi(a')v_{n, h}(a', \xi)da'\,(1-v_{n, h}),\\
		v_{n+1, h}(0, \xi)=\int_0^{a^+}\ga(a)v_{n, h}(a, \xi+ca)da.
	\end{cases}
\end{equation*}
On the other hand, recalling that $u^{n+1}$ satisfies \eqref{u^n} and due to the monotonicity \eqref{g} again along with $v_{n, h}\ge u^n$, we have 
\begin{eqnarray}
&&Mv_{n, h}(a, \xi)+\int_0^{a^+}K(a, a')\pi(a')v_{n, h}(a', \xi)da'\,(1-v_{n, h}(a, 
\xi))\nonumber\\
&\ge& Mu^{n}(a, \xi)+\int_0^{a^+}K(a, a')\pi(a')u^{n}(a', \xi)da'\,(1-u^{n}(a, \xi)).\nonumber
\end{eqnarray}
As the initial data satisfies, 
\begin{eqnarray}
	v_{n+1, h}(0, \xi)&=&\int_0^{a^+}\ga(a)v_{n, h}(a, \xi+ca)da\nonumber\\
	&\ge& \int_0^{a^+}\ga(a)u^n(a, \xi+ca)da=u^{n+1}(0, \xi), \;\forall\xi\in\R,\nonumber
\end{eqnarray}
the comparison principle applies and provides
$$ 
u^{n+1}(a, \xi)\le v_{n+1, h}(a, \xi), \;\forall (a, \xi)\in[0, a^+]\times\R.
$$
Now since for all $a\in[0, a^+]$ the function $\xi\to u^{n+1}(a, \xi)$ is nonincreasing, for all $a\in[0, a^+]$ and $\xi\in\R$, we get
\begin{eqnarray}
	|u^{n+1}(a, \xi+h)-u^{n+1}(a, \xi)|\le u^{n+1}(a, \xi)-u^{n+1}(a, \xi+h)&\le& (1-e^{-m h})v_{n+1, h}(
	a, \xi)\nonumber\\
	&\le& e^{m h}-1.
\end{eqnarray}
Hence, we have obtained that, for all $n\ge1$ and for all $h>0$, 
$$
|u^n(a, \xi+h)-u^n(a, \xi)|\le\min\{1, e^{m|h|}-1\}, \, \forall (a, \xi)\in[0, a^+]\times\R.
$$
As mentioned before, the case $h<0$ can be handled similarly and thus the result is desired.
\end{proof} 

Now the estimates provided in Lemma \ref{Lip} allows us to conclude that the limit $w$ obtained in Section \ref{existence} is globally Lipschitz continuous. This complete the desired regularity of $w$ with respect to $\xi$ as needed in Definition \ref{TW}. 
Thus we have obtained a continuous and bounded solution $w$ satisfying \eqref{w-equation}, which is a traveling wave solution of \eqref{u-equation}. This completes the proof of Theorem \ref{main1} for $c>c^*$.

Next, let us prove the existence of traveling waves for the critical wave speed $c=c^*$. To this aim, assume that $\{c_l\}_{l\in\N}\in(c^*, c^*+1)$ is a decreasing sequence satisfying $\lim\limits_{l\to\infty}c_l=c^*$. Following the argument after Proposition \ref{sub-super method}, for each $c_l$ there exists a traveling wave solution satisfying \eqref{w-equation}, denoted by $w_l$. Since $w_l(a, \cdot+\eta)$ is also a solution of \eqref{w-equation} for any $\eta\in\R$, we can assume that $w_l(0, 0)=\frac{1}{2}$. Recalling Lemma \ref{speed}, for any $c_l$, there exists $\la_l$ satisfying \eqref{s_0}. Observe that $c_l\to\la_l=\la_l(c_l)$ is continuous due to \eqref{s_0}. It follows that $\{\la_l\}_{l\in\N}$ is bounded due to $\{c_l\}_{l\in\N}\in(c^*, c^*+1)$ and thus we can choose $m>0$ large enough such that $m>\{\la_l\}_{l\in\N}$. Applying the argument in Lemma \ref{Lip} to $\{w_l\}_{l\in\N}$, we obtain that $\{w_l\}_{l\in\N}$ is equi-continuous with respect to $\xi$. Due to $\{w_l\}_{l\in\N}\in[0, 1]$, one see from the equation that $\{w_l\}_{l\in\N}$ is equi-continuous with respect to $a$. Now by Arzel\`a-Ascoli Theorem, there exists a subsequence of $\{w_l\}_{l\in\N}$, again denoted by $\{w_l\}_{l\in\N}$ satisfying $w_l\to w^*$ as $l\to\infty$ and that $\xi\to w^*(a, \xi)$ is globally Lipschitz continuous for any $a\in[0, a^+]$ and that $a\to w^*(a, \xi)\in W^{1, 1}(0, a^+)$ for any $\xi\in\R$. It is clear that $w^*(0, 0)=\frac{1}{2}$. In addition, by the same argument in Proposition \ref{sub-super method}, $w^*$ satisfies the following equation
	\begin{equation}
		\begin{cases}
			\frac{\partial w^*(a, \xi)}{\partial a}=\left(J\ast_\xi w^*-w^*\right)(a, \xi)+\int_0^{a^+}K(a, a')\pi(a')w^*(a', \xi)da'\,(1-w^*(a, \xi)),\\
			w^*(0, \xi)=\int_0^{a^+}\gamma(a)w^*(a, \xi+ca)da.
		\end{cases}
	\end{equation}
	
On the other hand, $0\le w^*\le 1$ and $\xi\to w^*(a, \xi)$ being nonincreasing for any $a\in[0, a^+]$ imply that the limits $w^*(a, \pm\infty)=w^\pm(a)$ exist for any $a\in[0, a^+]$. Moreover, the limits satisfy the equation \eqref{equi2}, together with the condition 
$$
w^-(0)\geq \frac{1}{2}\geq w^+(0).
$$
Hence Lemma \ref{LE-lim} applies and ensures that
$w^*(a, -\infty)=1$ and $w^*(a, \infty)=0$ uniformly in $[0, a^+]$. Hence the existence for $c=c^*$ is complete.

Now we set $I=\pi(a)u$ and then obtain that it is a traveling wave solution of \eqref{I-equation} with speed $c$. Finally, let us notice that $\pi(a^+)=0$ and $u$ being bounded implies that function $I$ vanishes at $a=a^+$. Hence Theorem \ref{main2} is complete.

\section{Spreading speeds}\label{S}

In this section, we shall prove Theorem \ref{SS}. 
\subsection{Outer spreading}
For Theorem \ref{SS}-(i), let $c>c^*$ be given and fixed. We now look for a super-solution of \eqref{IVP-u} with the form $v(t, a, x)=v_0e^{-\lambda(x-ct)}\phi(a)$ for some positive constant $v_0$ to be chosen later and $\lambda>0$ satisfying \eqref{s_0}. Here $\phi$ is defined in \eqref{phi_2} with $s=s_0$, where $s_0$ is defined in Section \ref{css}. Note that, for all $t\in\R, a\in(0, a^+)$ and $x\in\R$ one has
\begin{eqnarray}
&&\frac{\partial v}{\partial t}+\frac{\partial v}{\partial a}-v_0\phi(a)\left[\int_\R J(x-y)e^{-\lambda(y-ct)}dy-e^{-\lambda(x-ct)}\right]-\int_0^{a^+}K(a, a')\pi(a')v(a')da'\nonumber\\
&=&c\lambda v+\phi'(a)\frac{v}{\phi}-v\left[\int_\R J(x-y)e^{\lambda(x-y)}dy-1\right]-\int_0^{a^+}K(a, a')\pi(a')v(a')da'\nonumber\\
&=&0.\nonumber
\end{eqnarray}
Besides, again by \eqref{phi_2} one has
$$
v(t, 0, x)=v_0e^{-\lambda(x-ct)}\phi(0)=v_0e^{-\lambda(x-ct)}\int_0^{a^+}\gamma(a)\phi(a)da=\int_0^{a^+}\gamma(a)v(t, a, x)da.
$$
On the other hand, we know from \eqref{IVP-u} that
$$
\frac{\partial u}{\partial t}+\frac{\partial u}{\partial a}\le J\ast u-u+\int_0^{a^+}K(a, a')\pi(a')u(a')da'.
$$
Finally, since $u_0$ is compactly supported and $\inf_{a\in[0, a^+]}\phi(a)>0$, we can choose $v_0$ large enough such that
$$
v(0, a, x)=v_0e^{-\lambda x}\phi(a)\ge u_0(a, x), \;\text{ for all } x\in\R \text{ and }a\in(0, a^+)
$$
which implies that $v$ is a super-solution of \eqref{IVP-u}. Hence we end-up with the following estimate due to Lemma \ref{wcp}
$$
u(t, a, x)\le v(t, a, x)=v_0e^{-\lambda(x-ct)}\phi(a), \;\text{ for all } x\in\R,\, t\ge0 \text{ and }a\in(0, a^+).
$$
Now let $c_1$ be any real number such that $c_1>c>c^*$. Then we have
$$
\sup_{|x|\ge c_1t,\, 0<a<a^+}v(t, a, x)=v_0e^{-\lambda(c_1-c)t}\phi(a)\to0, \text{ as }t\to\infty,
$$
therefore as $c>c^*$ can be chosen arbitrary close to $c^*$, Theorem \ref{SS}-(i) holds true in the case $x\ge ct$ and consequently for $|x|\ge ct$ as well. This concludes the proof of Theorem \ref{SS}-(i).

\subsection{Hair-trigger effects}
In order to prove Theorem \ref{SS}-(ii), we first investigate the so-called hair trigger effect of \eqref{IVP-u}, that roughly speaking indicates the ability of the solutions of the system to become uniformly positive whatever the smallness of the non-zero and non-negative initial data.   
\begin{lemma}\label{HTE}
	Let $u=u(t, a, x)$ be the solution of \eqref{IVP-u} with the initial data $u_0=u_0(a, x)$.
	If there exist two constants $x_0\in\R$ and $\rho_0\in(0, 1)$ such that
	$$
	u_0(a, x)\ge\rho_0 \text{ for all $x\in [x_0-1,x_0+1]$ and }a\in(0, a^+),
	$$
	then for any $\rho\in(0, 1)$, there exists $T_{\rho_0}^\rho\ge0$ which is independent of $x_0$ such that
	\begin{equation}\label{hair-trigger}
		u(t, a, x)\ge\rho \text{ for all $x\in [x_0-1,x_0+1]$},\, t\ge T_{\rho_0}^\rho \text{ and }a\in(0, a^+).
	\end{equation}
\end{lemma}
\begin{proof}
	We consider the following auxiliary problem
	\begin{equation}\label{vvv}
		\begin{cases}
			\frac{\partial w}{\partial t}=J\ast w-w+\Phi_{\min}w(1-w), \\
			w(0, x)=\rho_0\mathbf{1}_{[x_0-1,x_0+1]}(x),
		\end{cases}
	\end{equation}
    where 
    \begin{equation}\label{Phi}
    	\Phi_{\min}:=\min_{a\in[0, a^+]}\int_0^{a^+}K(a, a')\pi(a')da'>0, \text{ due to Assumption \ref{Assump}-(ii) and (iv) }
    \end{equation} 
    and $\mathbf{1}_S$ denotes the indicator function of the set $S$. Then it is easy to see that the solution $w$ of 
    \eqref{vvv} is a sub-solution of \eqref{IVP-u} with $u_0(a, x)\ge\rho_0\mathbf{1}_{B_1(x_0)}(x)$.
    
    However, recall Xu et al. \cite[Lemma 4.1]{xu2021spatial} (see also Alfaro \cite[Theorem 2.6]{alfaro2017fujita} or Finkelschtein and Tkachov \cite[Theorems 2.5 and 2.7]{finkelshtein2018hair}), the equation \eqref{vvv} exhibits the hair trigger effect, that is there exists $T_{\rho_0}^\rho\ge0$ which is independent of $x_0$ such that
    $$
    w(t, x)\ge\rho \text{ for any }x\in [x_0-1,x_0+1],\, t\ge T_{\rho_0}^\rho.
    $$
    Now the comparison principles applies to conclude the desired result \eqref{hair-trigger}.
\end{proof}

\begin{corollary}\label{CORO}
Let $u=u(t, a, x)$ be the solution of \eqref{IVP-u} with non-negative initial data $u_0=u_0(a, x)\in C([0,a^+]\times\R)\setminus\{0\}$.
Then one has for each $t\ge a^+$ and for any $x\in\R$,
$$
\inf_{a\in (0,a^+)} u(t,a,x)>0.
$$
Due to the above property, Lemma \ref{HTE} implies that, if in addition $u_0\leq 1$, then 
$$
\lim_{t\to\infty} \sup_{a\in (0,a^+)}|u(t,a,x)-1|=0,\text{ locally uniformly for $x\in \R$.}
$$

\end{corollary}

\begin{proof}
First we define an operator $\mathcal K: L^1((0, a^+), X)\to L^1((0, a^+), X)$ as follows,
\begin{equation*}
	\mathcal K(u):=\int_0^{a^+}K(\cdot, a')\pi(a')u(a')da'\;(1-u), \quad u\in L^1((0, a^+), X).
\end{equation*}
Next solving the problem \eqref{IVP-u} along the characteristic line $a-t=c$, where $c\in\R$, we now derive the formula for a solution to \eqref{original}. For fixed $c\in\R$, we set $w(t)=u(t, t+c)$ for $t\in[\max(-c, 0), \infty)$. With $a=t+c$ one obtains for $t\in[\max(-c, 0), \infty)$ the equation
\begin{equation}\label{www}
	\partial_tw(t)=Tw-w+[\mathcal K(w(t))](t+c),
\end{equation}
where $T$ is defined in \eqref{op-K}. We first study the case $c\ge0$. Clearly, $w(0)=u(0, c)=u(0, a-t)=u_0(a-t)$. Considering the equation \eqref{www} with initial data $w(0)\ge0$ and $w(0)\not\equiv0$, we have $w(t)>0$ for $t>0$ by the strong comparison principle of the nonlinear nonlocal diffusion problem, due to $J(0)>0$ in Assumption \ref{Assump}-(iii). It follows that $u(t, a)>0$ for $a\ge t$. On the other hand, integrating \eqref{www} from $0$ to $t$, one obtains
$$
w(t)=e^{(T-I)t}w(0)+\int_0^{t}e^{(T-I)(t-s)}[\mathcal K(w(s))](s+c)ds,
$$ 
and thus
$$
u(t, a)=e^{(T-I)t}u_0(a-t)+\int_0^{t}e^{(T-I)(t-s)}[\mathcal K(u(s))](s+a-t)ds.
$$

Next we consider the case $c<0$. Integrating \eqref{www} from $-c$ to $t$, one gets 
$$
w(t)=e^{(T-I)(t+c)}w(-c)+\int_{-c}^te^{(T-I)(t-s)}[\mathcal K(u(s))](s+c)ds,
$$
and thus
$$
u(t, a)=e^{(T-I)a}u(t-a, 0)+\int_{t-a}^te^{(T-I)(t-s)}[\mathcal K(u(s))](s+a-t)ds.
$$
Thus now the solution to \eqref{original} reads as follows:
\begin{equation}\label{explicit}
	u(t, a)=
	\begin{cases}
		e^{(T-I)t}u_0(a-t)+\int_0^{t}e^{(T-I)(t-s)}[\mathcal K(u(s))](s+a-t)ds, &a\ge t,\\
		e^{(T-I)a}u(t-a, 0)+\int_{t-a}^te^{(T-I)(t-s)}[\mathcal K(u(s))](s+a-t)ds, &a<t.
	\end{cases}
\end{equation}
Next we plug the explicit formula \eqref{explicit} into $u(t, 0)$ to obtain 
\begin{eqnarray}
	u(t, 0)&=&\int_0^{t}\chi(a)\ga(a)\left[e^{(T-I)a}u(t-a, 0)+\int_{t-a}^te^{(T-I)(t-s)}[\mathcal K(u(s))](s+a-t)ds\right]da\nonumber\\
	&&+\int_t^{a^+}\chi(a)\ga(a)\left[e^{(T-I)t}u_0(a-t)+\int_0^{t}e^{(T-I)(t-s)}[\mathcal K(u(s))](s+a-t)ds\right]da,\label{u(t, 0)}
\end{eqnarray}
where $\chi(a)$ is the cutoff function satisfying $\chi(a)=1$ when $a\in(0, a^+)$ otherwise $\chi(a)=0$. Now we consider two cases.

\textbf{Case 1.} If $t<a^+$, \eqref{u(t, 0)} is written as follows:
\begin{eqnarray}
	u(t, 0)&=&\int_0^{t}\ga(a)\left[e^{(T-I)a}u(t-a, 0)+\int_{t-a}^te^{(T-I)(t-s)}[\mathcal K(u(s))](s+a-t)ds\right]da\nonumber\\
	&&+\int_t^{a^+}\ga(a)\left[e^{(T-I)t}u_0(a-t)+\int_0^{t}e^{(T-I)(t-s)}[\mathcal K(u(s))](s+a-t)ds\right]da.\label{one}
\end{eqnarray}
Since $e^{(T-I)t}u_0(a-t)>0$ for $a\ge t$ and $\ga\ge0$ for any $a\in [0, a^+]$ by Assumption \ref{Assump}-(i), the second term on the right hand of \eqref{one} must be positive. Thus we have $u(t, 0)>0$ which implies $u(t, a)>0$ for $a<t$ via \eqref{explicit}.

\textbf{Case 2.} If $t\ge a^+$, \eqref{u(t, 0)} is written as follows:
\begin{equation}\label{three}
	u(t, 0)=\int_0^{a^+}\ga(a)\left[e^{(T-I)a}u(t-a, 0)+\int_{t-a}^te^{(T-I)(t-s)}[\mathcal K(u(s))](s+a-t)ds\right]da.
\end{equation}
Let us claim that $u(t, 0, x):=[u(t, 0)](x)>0$ in $[a^+, \infty)\times\R$. By contradiction, suppose that there exists $(t_0, x_0)\in[a^+, \infty)\times\R$ such that $u(t_0, 0, x_0)=0$. Thus one obtains
\begin{equation}
	0\ge \int_0^{a^+}\ga(a)e^{-a}e^{Ta}u(t_0-a, 0, x_0)da,\nonumber
\end{equation}
where we used the fact that $e^{-a}$ and $e^{Ta}$ are commuting. By Assumption \ref{Assump}-(i) on $\ga$, we can find one point $b_0\in(0, a^+]$ such that $e^{Ta}u(t_0-b_0, 0, x_0)=0$. By definition, one has
$$
e^{Ta}u(t_0-b_0, 0, x_0)=\sum_{n=0}^\infty \frac{(a)^n}{n!}J^{*n}\ast u(t_0-b_0, 0, x_0),
$$
where $J^{*n}$ denotes the $n$-fold convolution of $K$; that is $J^{*n}=J\ast\cdots\ast J$, $n$ times. It follows that for each $n\in\N$,
$$
J^{*n}\ast u(t_0-b_0, 0, x_0)=0.
$$ 
However, by Assumption \ref{Assump}-(iii) on $J$, one has $J>0$ in $(-r, r)$ for some $r>0$, which implies that
$$
u(t_0-b_0, 0, x)=0, \text{ for all }x\in (x_0-nr, x_0+nr).
$$
Since $\R=\sum_{n=0}^\infty(x_0-nr, x_0+nr)$, we have $u(t_0-b_0, 0, \cdot)\equiv0$ in $\R$. 

Next replace $t_0$ by $t_0-b_0$ in \eqref{u(t, 0)}. If $t_0-b_0$ falls in $[0, a^+]$, by the argument as Case 1, one has $u(t_0-b_0, 0)>0$, which is a contradiction. Hence, $t_0-b_0$ must fall in $(a^+, \infty)$. Then by the same argument as Case 2, one can find $b_1\in(0, a^+]$ such that $u(t_0-b_0-b_1, 0)=0$. Now repeating the above process by induction, one can find a sequence $\{b_i\}_{i\ge0}$ such that $u(t_0-\sum_{i=0}^{\widehat M} b_i, 0)=0$ for any $\widehat M\ge0$. But we know every $b_i$ is in $(0, a^+]$, then there always exists a minimal $M_0>0$ such that $t_0-\sum_{i=0}^{M_0} b_i<a^+$. Then by Case 1, one has $u(t_0-\sum_{i=0}^{M_0} b_i, 0)>0$. 

Now consider \eqref{three} at $t=t_0-\sum_{i=0}^{M_0-1}b_i$, which is larger than or equal to $a^+$, we get a contradiction, since now the left hand side of \eqref{three} equals to zero, while the right hand side of \eqref{three} is larger than zero.

In summary, we cannot have $(t, x)\in(0, \infty)\times\R$ such that $u(t, 0, x)=0$, which implies that $u(t, 0, x)>0$ and thus $u(t, a)>0$ by \eqref{explicit}. Hence, the proof is complete.
\end{proof}

\subsection{Inner spreading}

In this section, we prove the inner spreading. To this aim, we consider the following auxiliary equation 
\begin{equation}\label{v}
	\begin{cases}
		\frac{\partial v}{\partial t}=J\ast v-v+\lambda_0v(1-Pv),\\
		v(0, x)=v_0(x),
	\end{cases}
\end{equation}
where $v_0$ denotes a Lipschitz continuous in $\R$ with compact support satisfying 
$$
v_0(x)\le\frac{1}{P} \text{ in }\R,
$$
and, recalling the definition of the function $\phi$ in \eqref{phi_2} with $s=s_0<0$, $P$ and $\la_0$ are the positive constants given by 
$$
P=\frac{\sup_{a\in(0, a^+)}\int_0^{a^+}K(a, a')\pi(a')\phi(a')da'}{\la_0}>0 \text{ and } \la_0:=-s_0>0 \text{ (see the proof of Lemma \ref{speed}) }
$$ 
Recall that $u_0=u_0(a,x)\in C^0([0,a^+]\times\R)\setminus\{0\}$ is non-negative, $u_0\leq 1$ and ${\rm supp}(u_0)$ is compact in $[0,a^+]\times \R$. Hence according to Corollary \ref{CORO}, one has there exists $\rho>0$ such that
$$
\tilde u_0(a,x):=u(a^+,a,x)\geq \rho,\;\;\forall a\in [0, a^+],\;x\in[-R, R].
$$ 
Next we fix $v_0:\R\to\R^+$ Lipschitz continuous such that
$$
{\rm supp}\,(v_0)\subset [-R,R],\;\;\sup_{x\in\R} v_0(x)\leq \frac{1}{P}\text{ and }\left(\sup_{a\in [0,a^+]}\phi(a)\right)\left(\sup_{x\in \R}v_0(x)\right)\leq \rho.
$$
so that the function $\underline{\textbf u}_0\in C([0, a^+]\times\R)$ given by
$$
\underline{\textbf u}_0(a, x):=\phi(a)v_0(x),\;\forall (a,x)\in [0,a^+]\times\R,
$$
satisfies
$$
\underline{\textbf u}_0(a, x)\le u(a^+,a, x),\;\;\forall (a,x)\in [0,a^+]\times\R.
$$
Next set $\underline{\textbf u}(t, a, x)=\phi(a)v(t, x)$, where $v=v(t, x)$ is the solution of \eqref{v} with the initial data $v_0$. Next, we verify that $\underline{\textbf u}$ is a sub-solution of \eqref{IVP-u}. Indeed, we have
\begin{eqnarray}
	\frac{\partial\underline{\textbf u}}{\partial t}+\frac{\partial\underline{\textbf u}}{\partial a}&=&J\ast \underline{\textbf u}-\underline{\textbf u}+\la_0\underline{\textbf u}(1-Pv)+s_0\underline{\textbf u}+v\int_0^{a^+}K(a, a')\pi(a')\phi(a')da'\,\nonumber\\
	&=&J\ast \underline{\textbf u}-\underline{\textbf u}+v\int_0^{a^+}K(a, a')\pi(a')\phi(a')da'-P\la_0\underline{\textbf u}v\nonumber\\
	&\le&J\ast \underline{\textbf u}-\underline{\textbf u}+v\int_0^{a^+}K(a, a')\pi(a')\phi(a')da'\, (1-\underline{\textbf u}),
\end{eqnarray}
and by \eqref{overline U} and the choice of $v_0$, one has
\begin{eqnarray}
	&&\underline{\textbf u}(t, 0, x)=\phi(0)v(t, x)=\int_0^{a^+}\ga(a)\phi(a)da\,v(t, x)=\int_0^{a^+}\ga(a)\underline{\textbf u}(t, a, x)da,\nonumber\\
	&&\underline{\textbf u}(0, a, x)=\phi(a)v_0(x)\le u_0(a, x).\nonumber
\end{eqnarray}
Now define
$$
c_0:=\inf_{\la>0}\frac{\int_{\R}J(y)e^{\lambda y}dy-1+\la_0}{\la}.
$$
Note that $c_0>0$ since $J$ is symmetric. Then let us recall a spreading result for equation \eqref{v}, from Ducrot and Jin \cite[Lemma 3.6]{ducrot2022spreading}.
\begin{lemma}
	let $v=v(t, x)$ be the solution of \eqref{v} supplemented with a continuous initial data $0\le v_0(\cdot)\le\frac{1}{P}$ and $v_0\not\equiv0$ with compact support. Let us further assume that $v$ is uniformly continuous for all $t\ge0, x\in\R$. Then one has
	$$
	\lim_{t\to\infty}\sup_{|x|\le ct}\left|v(t, x)-\frac{1}{P}\right|=0, \;\forall 0<c<c_0.
	$$ 
\end{lemma}
Next, let us show $c_0=c^*$. Define the function $H(\la), \la>0$ as follows,
$$
H(\la):=\frac{\int_{\R}J(y)e^{\lambda y}dy-1+\la_0}{\la}.
$$
Differentiating $H$ with respect to $\la$ and setting $\la^*=\la(c^*)$, one obtains by \eqref{c*}
$$
H'(\la)=\frac{\int_{\R}J(y)e^{\lambda y}(\la y-1)dy-\int_{\R}J(y)e^{\lambda^* y}(\la^*y-1)dy}{\la^2}.
$$		
Next, when $H'(\la)=0$, one obtains that $\la$ satisfies
$$
\int_{\R}J(y)e^{\lambda y}(\la y-1)dy=\int_{\R}J(y)e^{\lambda^* y}(\la^*y-1)dy,
$$
which implies $\la=\la^*$ due to the monotonicity of $\int_{\R}J(y)e^{\lambda y}(\la y-1)dy$ with respect to $\la$ and thus $c_0=c^*$. Hence now for any $0\leq c<c^*$ we have 
\begin{equation}\label{u}
\lim_{t\to\infty}\inf_{|x|\le ct, 0<a<a^+}u(t+a^+, a, x)\ge\lim_{t\to\infty}\inf_{|x|\le ct, 0<a<a^+}\phi(a)v(t, x)\ge\frac{\min_{a\in[0, a^+]}\phi(a)}{P}=:\rho_0>0.
\end{equation}
Moreover, one may assume that $\rho_0<1$. Thus \eqref{u} implies that for any $c\in [0,c^*)$, there exists $T>0$ such that 
$$
u(t, a, x)\ge\frac{\rho_0}{2} \text{ for }t\ge T,\, |x|\le ct \,\text{ and }a\in(0, a^+).
$$
Now for any $\rho\in(0, 1)$, applying Lemma \ref{HTE} to \eqref{IVP-u} yields that there exists $T_{\rho_0}^\rho>0$ such that
$$
u(t+T_{\rho_0}^\rho, a, x)\ge\rho, \text{ for }t\ge T,\, |x|\le ct\, \text{ and }a\in(0, a^+),
$$
which implies that
$$
\inf_{|x|\le ct-cT_{\rho_0}^\rho, 0<a<a^+}u(t, a, x)\ge\rho, \text{ for }t\ge T+T_{\rho_0}^\rho.
$$
For any $\ep\in(0, c)$, there exists a constant $T'\ge T_{\rho_0}^\rho$ such that $\ep T'\ge cT_{\rho_0}^\rho$. Then we have that $ct-cT_{\rho_0}^\rho\ge (c-\ep)t$ and 
$$
\inf_{|x|\le (c-\ep)t, 0<a<a^+}u(t, a, x)\ge\rho, \text{ for }t\ge T'.
$$
Now since $\rho\in(0, 1)$ is arbitrary close to $1$, we obtain 
$$
\lim_{t\to\infty}\inf_{|x|\le (c-\ep)t, 0<a<a^+}u(t, a, x)=1.
$$
Due to the arbitrariness of $\ep$, hence we have the following result.
\begin{theorem}
	let $u=u(t, a, x)$ be the solution of \eqref{IVP-u} supplemented with a continuous initial data $0\le u_0\le1$ and $u_0\not\equiv0$ with $u_0$ being compactly supported in $[0, a^+]\times\R$, then one has
	$$
	\lim_{t\to\infty}\sup_{|x|\le ct, 0<a<a^+}\left|u(t, a, x)-1\right|=0, \;\text{ for all } c\in(0, c^*).
	$$ 
\end{theorem}
This proves Theorem \ref{SS}-(ii).

\bibliography{hulk}
\bibliographystyle{plain}
\addcontentsline{toc}{section}{\refname}

\end{document}